\documentclass[12pt]{amsart}
\usepackage{amssymb}
\usepackage{float}

\usepackage{amsmath}
\numberwithin{equation}{section} 
\usepackage{amsfonts}
\usepackage{graphicx} 
\usepackage{amsthm}
\usepackage{color}
\usepackage{enumerate}
\usepackage{hyperref}
\usepackage{marginnote}
\usepackage{enumitem} 
\usepackage{tikz-cd}

\usepackage{kotex}

\setlist[enumerate]{itemsep=0.2em, topsep=0.25em}

\usepackage[a4paper,left=3cm,right=3cm,top=2.5cm,bottom=3cm]{geometry}

\theoremstyle{definition}
\newtheorem{theorem}{Theorem}[section]
\newtheorem{introthm}{Theorem}

\newtheorem{cor}[theorem]{Corollary}
\newtheorem{prop}[theorem]{Proposition}
\newtheorem{lemma}[theorem]{Lemma}

\theoremstyle{definition}
\newtheorem{example}[theorem]{Example}
\newtheorem{question}[theorem]{Question}
\newtheorem{definition}[theorem]{Definition}
\newtheorem{rem}[theorem]{Remark}

\newcommand{\str}{\operatorname{str-exp}}

\newcommand{\ext}{\operatorname{ext}}

\newcommand{\N}{\mathbb{N}}

\newcommand{\Lin}{\mathcal{L}}

\newcommand{\F}{\mathcal{F}}
\newcommand{\eps}{\varepsilon}

\newcommand{\vertiii}[1]{{\left\vert\kern-0.25ex\left\vert\kern-0.25ex\left\vert #1 
		\right\vert\kern-0.25ex\right\vert\kern-0.25ex\right\vert}}

\DeclareMathOperator{\supp}{supp}
\DeclareMathOperator{\dist}{dist}

\DeclareMathOperator{\sgn}{sgn}

\DeclareMathOperator{\co}{co}

\DeclareMathOperator{\NA}{NA}

\DeclareMathOperator{\SNA}{SNA}

\DeclareMathOperator{\ASE}{ASE}
\DeclareMathOperator{\SE}{SE}

\DeclareMathOperator{\LipSE}{LipSE}

\newcommand{\Lip}{{\mathrm{Lip}}_0}
\newcommand{\Mol}{\operatorname{Mol}}
\renewcommand{\subset}{\subseteq}
\newcommand{\free}{\mathcal{F}}

\title{On absolute strong exposure for Lipschitz maps}
\author[Choi]{Geunsu Choi}
\address[Choi]{Department of Mathematics Education, Sunchon National University, 57922 Jeonnam, Republic of Korea \newline
\href{http://orcid.org/0000-0002-4321-1524}{ORCID: \texttt{0000-0002-4321-1524}}}
\email{\texttt{gschoi@scnu.ac.kr}}

\author[Jung]{Mingu Jung}
        \address[Jung]{Department of Mathematics \& Research Institute for Natural Sciences, Hanyang University, 04763 Seoul, Republic of Korea \newline
\href{https://orcid.org/0000-0003-2240-2855}{ORCID: \texttt{0000-0003-2240-2855}}}
\email{mingujung@hanyang.ac.kr}

\date{\today}

\begin{document}

\begin{abstract}
We introduce strongly exposing Lipschitz maps, a vector-valued extension of Weaver's peaking functions and a nonlinear analogue of absolutely strongly exposing operators. Our main result shows that a Lipschitz map is strongly exposing if and only if its canonical linearization is absolutely strongly exposing. This equivalence serves as a bridge between the linear and Lipschitz settings and enables us to transfer several results from the former to the latter. As applications, we establish norm-denseness and residuality results for strongly exposing Lipschitz maps, obtain an isomorphic characterization related to the denseness of strongly norm-attaining Lipschitz maps. We also investigate weak sequential denseness of strongly exposing Lipschitz maps. In particular, we prove that this property holds whenever the derived set of the underlying metric space is finite, while further examples show that, unlike for strongly norm-attaining Lipschitz maps, weak sequential denseness may fail beyond trivial cases.
\end{abstract}

\maketitle

\section{Introduction}

The study of norm-attaining operators has been one of the central topics in Banach space theory since the Bishop-Phelps theorem established the denseness of norm-attaining functionals in the dual of every Banach space. In contrast with the case of linear functionals, the corresponding denseness problem for bounded linear operators is considerably more delicate and depends in an essential way on the geometry of the spaces involved.

A fundamental contribution to this problem was made by Bourgain, who introduced the notion of \emph{absolutely strongly exposing operators}. This notion provides a strengthened form of norm-attainment. Roughly speaking, an operator is absolutely strongly exposing if it attains its norm at a point in such a way that every sequence which almost attains the norm is forced, up to the natural scalar ambiguity and after passing to subsequences, to converge to that point. Thus, absolutely strongly exposing operators ensure not only the fact that the norm is attained, but also the stability and uniqueness-like behavior of norm-attainment.

The purpose of this paper is to develop a Lipschitz counterpart of this Bourgain-type theory. Throughout the paper, $X,Y$ denote real Banach spaces, and $M$ denotes a pointed metric space with distinguished point $0$. Lipschitz maps from $M$ into $Y$ provide a natural nonlinear analogue of bounded linear operators.

Recall that a bounded linear operator $T$ in the space of all bounded linear operators $\Lin(X,Y)$ \emph{attains its norm} if the norm 
$$
\|T\| := \sup \{\|Tx\|: x \in X, \, \|x\|\leq 1\}
$$
is attained at some point in the closed unit ball $B_X$. The corresponding notion in the Lipschitz setting is that of \emph{strong norm attainment}: The Banach space of all Lipschitz maps $f:M\to Y$ vanishing at $0$, endowed with the norm
\[
\|f\|:=\sup \left\{ 
\frac{\|f(x)-f(y)\|}{d(x,y)} : {x\neq y\in M}\right\} ,
\]
is denoted by $\Lip(M,Y)$. A map $f\in\Lip(M,Y)$ is said to \emph{strongly attain its norm} if there exist distinct points $p,q\in M$ such that
\[
\|f\|
=
\frac{\|f(p)-f(q)\|}{d(p,q)}.
\]
The set of all such maps is denoted by $\SNA(M,Y)$. If $Y=\mathbb{R}$, we simply write $\Lip (M)$ and $\SNA(M)$ in place of $\Lip (M,\mathbb{R})$ and $\SNA(M,\mathbb{R})$, respectively.

A fundamental tool in this setting is the \emph{Lipschitz-free space} $\mathcal{F}(M)$, which provides a linearization of Lipschitz maps. More precisely, let $\delta_M$ be the canonical evaluation embedding given by $\delta_M (x) (f) : = f(x)$ for every $x \in M$ and $f \in \Lip (M)$. Every Lipschitz map $f:M\to Y$ admits a unique canonical linearization: 
\[
\begin{tikzcd}[column sep=large,row sep=large]
M
  \arrow[r,"f"]
  \arrow[d,"\delta_M"']
&
Y
  \arrow[d,"\delta_Y"]
\\
\mathcal F(M)
  \arrow[r,"\widehat f"']
  \arrow[ur,"T_f"']
&
\mathcal F(Y)
\end{tikzcd}
\]
where $T_f:\mathcal{F}(M)\to Y$ and $\widehat{f}:\mathcal{F}(M)\to\mathcal{F}(Y)$ are the unique bounded linear maps making the diagram commute. Moreover,
\[
\|f\| = \|T_f\|= \|\widehat{f}\|.
\]
In what follows, whenever the underlying metric space is clear from the context, we write $\delta_x$ instead of $\delta_M(x)$. Observe also that the strong norm-attainment of $f$ corresponds to norm-attainmnt of $T_f$ at a \emph{molecule} 
\[
m_{p,q}:=\frac{\delta_p-\delta_q}{d(p,q)}
\]
for $p \neq q$ in $M$. 
 The set of all molecules $m_{p,q}$ is denoted by $\Mol(M)$. For simplicity, we use the notation 
\[
f(m_{p,q}):=\frac{f(p)-f(q)}{d(p,q)}
\]
for $f \in \Lip(M,Y)$ and $m_{p,q}\in\Mol(M)$.

Recall that an operator $T \in \mathcal{L}(X,Y)$ is \emph{absolutely strongly exposing} if there exists $x_0 \in B_X$ such that every sequence $(x_n) \subseteq B_X$ with $\|Tx_n\|\to\|T\|$ admits a subsequence converging to $x_0$ or $-x_0$. We write $T \in \ASE(X,Y)$. Motivated by this notion, we introduce the following Lipschitz analogue: besides strong-norm attainment at a molecule, we additionally require every sequence of almost norming molecules to converge, up to sign, to the norming molecule.

\begin{definition} A Lipschitz map $f \in \Lip(M,Y)$ is said to
\textit{strongly expose} at $m_{p,q} \in \Mol(M)$, and write by $f \in \LipSE(M,Y)$, if whenever a sequence $(m_{p_n,q_n}) \subseteq \Mol(M)$ satisfies that $\|f(m_{p_n,q_n})\| \to \|f\|$, there exists a subsequence $(m_{p_{k_n},q_{k_n}})$ such that $ (m_{p_{k_n},q_{k_n}})$ converges to either $m_{p,q}$ or $-m_{p,q}$. 
\end{definition}

If $Y=\mathbb{R}$, we simply write $\LipSE(M)$ for $\LipSE(M,\mathbb{R})$. 
The restriction to molecules in the above definition is natural, since it is well known that every strongly exposed point of $B_{\mathcal{F}(M)}$ is a molecule \cite[Lemma 4.1]{GPR}.

In the scalar-valued setting, strongly exposing Lipschitz maps coincide with the classical peaking functions in the sense of Weaver \cite[Definition 2.4.1]{W}. Consequently, the class $\LipSE(M,Y)$ may be viewed as a vector-valued extension of peaking functions. Our first main result shows that this notion is precisely the Lipschitz counterpart of absolute strong exposure for linear operators. \begin{introthm}\label{theorem:Lip-linear}
Let $f\in \Lip (M,Y)$ be given. Then the following are equivalent.
\begin{enumerate}
    \itemsep0.25em
    \item[\textup{(a)}] $f \in \LipSE (M, Y)$; 
    \item[\textup{(b)}] $\widehat{f} \in \ASE (\free (M), \free (Y))$;
    \item[\textup{(c)}] $T_f \in \ASE (\free (M), Y)$.
\end{enumerate}
\end{introthm}

Let us comment that this equivalence is not only of independent interest but also the main tool in the sequel; it will be used repeatedly to derive several results on $\LipSE(M,Y)$ from corresponding results on absolutely strongly exposing linear operators on $\mathcal F(M)$.

We then investigate norm-denseness and residuality of strongly exposing Lipschitz maps. Extending known results for strongly norm-attaining Lipschitz maps, we obtain several density theorems for $\LipSE(M,Y)$ under geometric assumptions on $\mathcal F(M)$ and on the range space $Y$. We also provide an isomorphic characterization of the universal denseness phenomenon, answering a question raised in \cite{CGMR}.

Finally, we study weak sequential denseness of strongly exposing Lipschitz maps. While it is known that $\SNA(M)$ is weakly sequentially dense
in $\Lip(M)$ for every metric space $M$ (\cite[Theorem 4.1]{CCGMR}), the class $\LipSE(M)$ may even be trivial when $M$ is a length space (see (\ref{iff_strepx_Z2}) in Section \ref{section:fundamental}). Nevertheless, weak sequential denseness remains valid for a surprisingly large class of metric spaces. Our second main result shows that it holds whenever the set of accumulation points of the metric space is finite.

\begin{introthm}\label{theorem:derived_set}
If the derived set $M'$ is finite (possibly empty), then $\LipSE(M)$ is weakly sequentially dense in $\Lip(M)$.
\end{introthm}

This theorem is one of the positive results showing that weak sequential denseness of strongly exposing Lipschitz maps often parallels the corresponding theory for strongly norm-attaining Lipschitz maps. However, the analogy is not complete. We also obtain further denseness results (see Proposition \ref{prop:isometric_copy_[0,1]} and Theorem \ref{theorem:peak-ball-weak-dense}), while exhibiting metric spaces, even among non-length spaces, for which weak sequential denseness fails (see Corollary \ref{cor:length-finite}, Theorem \ref{theorem:compact-proper-length} and Theorem \ref{theorem:str-exp-compact} for instance). Thus, in contrast with the case of strong norm-attainment, weak sequential exhibits a substantially different behavior from the corresponding theory of strongly norm-attaining Lipschitz maps.

\section{Fundamental results on strongly exposing Lipschitz maps}\label{section:fundamental}

In this section, we concentrate on fundamental results concerning strongly exposing Lipschitz maps aforementioned in the introduction, starting with the proof of Theorem \ref{theorem:Lip-linear}.

\subsection{Correspondence with bounded linear operators}

In order to prove Theorem \ref{theorem:Lip-linear}, we need the following preliminary lemma which allows that, whenever $B_X = \overline{\co} (V)$, the verification of the ASE property may be reduced to sequences in the generating set $V$. 

\begin{lemma}\label{lemma:co-extension}
Let $X$ and $Y$ be Banach spaces with $B_X = \overline{\co} (V)$ for some $V \subseteq S_X$. Let $T \in \Lin(X,Y)$ be such that $\|Tv_0\|=\|T\|=1$ for some $v_0 \in V$.  
Suppose that for every $\eps>0$, there exists $\delta>0$ such that whenever $v \in V$ satisfies that $\|Tv\| > 1-\delta$, there exists $\theta \in \{-1,1\}$ such that $\|v-\theta v_0\|<\eps$. Then $T \in \ASE(X,Y)$ at $v_0$.
\end{lemma}

\begin{proof}
Given $0<\eps<1$, find $0<\delta<\min\{\eps,1-\eps\}$ such that whenever $v \in V$ satisfies that $\|Tv\|>1-\delta$, we have that $\|v-v_0\| < \varepsilon$ or $\|v+v_0\| < \varepsilon.$

Suppose that $x \in B_X$ satisfies $\|Tx\| > 1 - \delta^2$. Choose $z \in \co V$ such that $\|x-z\|<\delta$ and $\|Tz\| > 1-\delta^2$, and write $z =  \sum_{i=1}^n \lambda_i x_i$, where $\lambda_1,\ldots,\lambda_n \geq0$ with $\sum_{i=1}^n \lambda_i=1$ and $x_1,\ldots,x_n \in V$. Let $y_0^* \in S_{Y^*}$ such that $y_0^*(Tz) = \|Tz\| > 1-\delta$. Let $\theta \in \{-1,1\}$ be such that $y_0^*(T (\theta v_0) ) \geq 0$. Define
$$
G := \{i \in \{1, \ldots, n\} : y_0^*(Tx_i) > 1- \delta\},
$$
$$
B := \{i \in \{1, \ldots, n\} : y_0^*(Tx_i) \leq 1 - \delta\}.
$$
It follows that if $i \in G$, then we have either $\|x_i- \theta v_0\| < \eps$ or $\|x_i+ \theta v_0\|< \eps$. Note that if $\|x_i + \theta v_0 \| < \eps$, then 
\[
\varepsilon<1-\delta < y_0^* (Tx_i)  \leq y_0^* (T (x_i + \theta v_0) ) \leq \|x_i + \theta v_0\| < \varepsilon,
\]
which is a contradiction. Thus, $\|x_i - \theta v_0\| < \varepsilon$ for each $i \in G$.

Now, we have
\begin{align*}
1- \delta^2 < \|Tz\| = y_0^*(Tz) &\leq \sum_{i \in G} \lambda_i y_0^*(Tx_i) + \sum_{i \in B} \lambda_i y_0^*(Tx_i) \\
&\leq \sum_{i \in G} \lambda_i + (1-\delta) \sum_{i \in B} \lambda_i = 1 - \delta \sum_{i \in B} \lambda_i,
\end{align*}
we implies that $\sum_{i \in B} \lambda_i < \delta$. Thus we obtain
\begin{align*}
\|x- \theta v_0 \| &\leq \|x-z\| + \sum_{i \in G} \lambda_i \|x_i - \theta v_0 \| + \sum_{i \in B} \lambda_i \|x_i - \theta v_0 \| \\
&< \delta + \eps  \sum_{i\in G} \lambda_i + 2\sum_{i \in B} \lambda_i    \\
&< \delta+\eps + 2 \delta < 4 \varepsilon.
\end{align*}
This completes the proof. 
\end{proof}

Before proving Theorem \ref{theorem:Lip-linear}, recall the canonical \emph{barycenter map} $\beta_Y : \mathcal{F}(Y) \to Y$ which is the unique contractive linear operator satisfying $\beta_Y (\delta_Y(y)) = y$ for every $y \in Y$ (\cite[p.~124]{GK}).

\begin{proof}[Proof of Theorem \ref{theorem:Lip-linear}]
(a)$\implies$(b). Suppose that $f \in \LipSE(M,Y)$ at $m_{p,q}$, and let $(m_{p_n,q_n}) \subseteq \Mol (M)$ be such that $ \|\widehat{f}(m_{p_n,q_n})\| \to \|\widehat{f}\| = \|f\|$. Note that 
\begin{equation}\label{eq:hat_f_molecule}
    \|\widehat{f}(m_{p_n,q_n})\|= \left\| \frac{\delta_{f(p_n)} - \delta_{f(q_n)}}  {d(p_n,q_n)} \right\| =\frac{\|f(p_n)-f(q_n)\|}  {d(p_n,q_n)} = \|f(m_{p_n,q_n})\| \to \|f\|.
\end{equation}
By the strong exposure of $m_{p,q}$, the sequence $(m_{p_n,q_n})$ admits a convergent subsequence converging to $m_{p,q}$ or $m_{q,p}$.
Applying Lemma \ref{lemma:co-extension} to $\widehat{f}:\mathcal{F}(M)\to\mathcal{F}(Y)$ with $V = \Mol(M)$, we conclude that $\widehat{f} \in \ASE(\free(M),\free(Y))$ at $m_{p,q}$. 

(b)$\implies$(c). This follows immediately from the existence of the barycenter map $\beta_Y : \free (Y) \to Y$ and $T_f = \beta_Y \circ \widehat{f}$. In fact, if $\mu \in B_{\mathcal{F}(M)}$, then  
\[
\|T_f (\mu)\| = \|(\beta_Y \circ \widehat{f} ) (\mu)\| \leq \|\widehat{f} (\mu)\| \leq \|\widehat{f}\|=\|f\|. 
\]
Thus, if $\|T_f (\mu_n)\| \to \|T_f\|$ for some $(\mu_n)$ in $B_{\free(M)}$, then $\| \widehat{f} (\mu_n)\| \to \|\widehat{f}\|$. Therefore, if $\widehat{f}$ is absolutely strongly exposing, then so is $T_f$.

(c)$\implies$(a). suppose that $\widehat{f} \in \ASE(\free(M),\free(Y))$ at $\mu_0 \in S_{\free(M)}$. If $(m_{p_n,q_n}) \subseteq \Mol(M)$ is a sequence such that $ \|f(m_{p_n,q_n})\| \to \|f\|$, then the calculation in \eqref{eq:hat_f_molecule} shows that $\|\widehat{f} (m_{p_n,q_n}) \|\to \|\widehat{f}\|$. It follows that there exists $\theta \in \{-1,1\}$ and a subsequence $(m_{p_{k_n},q_{k_n}})$ of $(m_{p_n,q_n})$ such that $m_{p_{k_n}, q_{k_n}} \to \theta  \mu_0$. As $\Mol(M)$ is sequentially norm closed (see \cite[Corollary 2.14]{GPPR}), $\mu_0 = m_{p,q}$ for some $p, q \in M$, with $p_{k_n} \to p$ and $q_{k_n} \to q$. This proves that $f \in \LipSE (M,Y)$ at $m_{p,q}$. 
\end{proof}

\subsection{Metric characterizations}

Recall from \cite{IKW} that a pair of distinct points $p,q \in M$ is said to have \textit{property \textup{(Z)}} if for every $\eps>0$, there exists $r \in M \setminus \{p,q\}$ such that
$$
d(p,r)+d(r,q) - d(p,q) \leq \eps \min \{ d(p,r), d(r,q)\}.
$$
In \cite[Theorem 5.4]{GPR}, it is shown that for a complete metric space $M$ and $p \neq q$ in $M$, 
\begin{equation}\label{iff_strexp_Z}
m_{p,q} \in \str (B_{\mathcal{F}(M)}) \,\, \iff \,\, \text{$(p,q)$ does not have the property \textup{(Z)}}.
    \end{equation}
Recall also that $M$ is said to have the property \textup{(Z)} if each pair of distinct points of $M$ has the property \textup{(Z)}. It is proved in \cite[Theorem 1.5]{AM} that property \textup{(Z)} is equivalent to $M$ being a length space. Combining this with \eqref{iff_strexp_Z}, we observe that 
\begin{align}\label{iff_strepx_Z2}
    \LipSE(M) = \{0\}  &\iff  \text{$M$ has the property \textup{(Z)}} \iff \text{$M$ is length},
\end{align}
since strongly exposing Lipschitz functions must strongly attain their norm at a pair $(p,q)$ such that $m_{p,q} \in \text{str-exp}(B_{\F(M)})$ by definition. This equivalence admits a natural vector-valued extension.

\begin{lemma}\label{lem:easy_lemma}
If $f \in \LipSE (M,Y)$ at $m_{p,q}$ for some $(p,q)\in \widetilde{M}$ and $y^* \in S_{Y^*}$ satisfies that $y^* (T_f (m_{p,q})) = \|f\|_L$, then $y^* \circ f \in \LipSE (M)$.    
\end{lemma}

\begin{proof}
    By Theorem \ref{theorem:Lip-linear}, $T_f \in \ASE (\mathcal{F}(M), Y)$ at $m_{p,q}$. Thus, \cite[Lemma 1.2]{JMR23} (see also \cite[Proposition 3.14]{CGMR}) shows that $T_{y^* \circ f} = T_f^* y^* \in \SE (\mathcal{F}(M))$. It follows again from Theorem \ref{theorem:Lip-linear} that $y^* \circ f \in \LipSE(M)$.
\end{proof}

\begin{prop}
    Let $M$ be a complete metric space. The following are equivalent.
\begin{enumerate}

\item[\textup{(a)}] $M$ is length.
\item[\textup{(b)}] $\LipSE(M,Y) = \{0\}$ for every Banach space $Y$.
\item[\textup{(c)}] There exists a nontrivial Banach space $Y$ such that $\LipSE(M,Y) = \{0\}$.
\end{enumerate}
\end{prop}

\begin{proof}
The implication (c)$\implies$(a), which is the only non-trivial one, follows from \eqref{iff_strepx_Z2} and Lemma \ref{lem:easy_lemma}. 
\end{proof}

\subsection{Denseness results from the geometry of $\F(M)$}\label{section:denseness}

We start with a direct consequence of Theorem \ref{theorem:Lip-linear}. Several situations in which the set of absolutely strongly exposing operators is dense are recorded in \cite[Proposition 4.2]{CGMR}. By the correspondence established in Theorem \ref{theorem:Lip-linear}, we obtain the following denseness result for strongly exposing Lipschitz maps.

\begin{cor}\label{cor:universal-domain-dense}
Suppose that one of the following holds:
\begin{enumerate}
\item[\textup{(i)}] $\mathcal{F}(M)$ has the \textup{RNP}.
\item[\textup{(ii)}] $B_{\mathcal{F}(M)}$ is the closed convex hull of a set of uniformly strongly exposing points.
\item[\textup{(iii)}] $\mathcal{F}(M)$ has property quasi-$\alpha$.
\end{enumerate}
Then, $\LipSE(M,Y)$ is dense in $\Lip(M,Y)$ for every Banach space $Y$.
\end{cor}

\begin{proof}
In each case, it is known that the set of $\ASE (\mathcal{F}(M),Y)$ is dense for every Banach space $Y$. The conclusion now follows from Theorem \ref{theorem:Lip-linear}.
\end{proof}

A useful method used in \cite{CGMR} to obtain ASE operators is to consider bounded
linear operators that attain their norm at strongly exposed points. Concretely, if
$T \in \Lin(X, Y )$ attains its norm at $x_0 \in \str (B_X)$ and $\varepsilon > 0$, then there exists
$S \in \ASE (X, Y )$ such that $\|Sx_0\|=\|S\|$ and $\|S-T\|  < \varepsilon$. This observation can be applied to the setting of Lipschitz functions as follows: 

\begin{lemma}\label{lemma:LipSE-ASE}
    If $f \in \SNA(M,Y)$ at $m_{p,q}\in\str( B_{\mathcal{F}(M)})$ and $\varepsilon>0$, then there exists $g \in \LipSE (M, Y)$ such that $\|g (m_{p,q})\| = \|g\|$ and $\|f-g\|<\varepsilon$.
\end{lemma}

\begin{proof}
If $f \in \SNA(M,Y)$ at $m_{p,q} \in \str(B_{\mathcal{F}(M)})$, then it follows that $T_f \in \NA(\mathcal{F}(M),Y)$ at $m_{p,q}$. Given $\eps>0$, the proof in \cite[Proposition 3.14]{CGMR} shows that there exists $T_g \in \ASE(\mathcal{F}(M),Y)$ which strongly exposes $m_{p,q}$ for some $g \in \Lip(M,Y)$ with $\|T_g-T_f\|<\eps$. By Theorem \ref{theorem:Lip-linear} and the isometric correspondence between $\Lip(M,Y)$ and $\Lin(\mathcal{F}(M),Y)$, we obtain $g \in \LipSE(M,Y)$ and $\|g-f\|<\eps$.
\end{proof}

Recall that $f \in \Lip (M, Y)$ is \emph{local} if for every $\varepsilon>0$ there exist $x,y \in M$ such that $0<d(x,y)<\varepsilon$ and $\|f (m_{x,y} ) \| > \|f\|_L - \varepsilon$. Notice that if $f \in \LipSE (M, Y)$, then $f$ is a non-local Lipschitz function. Conversely, it is known that if $M$ is compact and $f \in \Lip (M,Y)$ is non-local, then $T_f$ attains its norm at a strongly exposed molecule (see \cite[Lemma 3.13]{CGMR}). Therefore, the following is an immediate consequence of Lemma \ref{lemma:LipSE-ASE}.

\begin{cor}\label{cor:non-local}
If $f \in \Lip (M,Y)$ is a non-local Lipschitz function and $\varepsilon>0$, then there exists $g \in \LipSE (M,Y)$ such that $\|f-g\|_L < \varepsilon$. In particular, 
    \[
    \{f \in \Lip (M,Y) : \text{$f$ is non-local}\} \subseteq \overline{\LipSE (M,Y)}.
    \]
\end{cor}

\begin{rem}
As mentioned above, we have 
    \[\begin{tikzcd}
	{f \in \text{LipSE}(M)} & {f \text{ is non-local}} & {f\in\SNA(M) \text{ at } m_{p,q}\in \text{str-exp}(B_{\mathcal{F}(M)})}
	\arrow["\textup{($*$)}", Rightarrow, from=1-1, to=1-2]
	\arrow["\textup{($**$)}", Rightarrow, from=1-2, to=1-3]
\end{tikzcd}\]
where ($*$) holds for a complete metric space while ($**$) holds when $M$ is compact.

Neither implication ($*$) nor ($**$) is reversible in general. For ($*$), consider any uniformly discrete metric space $M$ such that $\LipSE (M) \neq \Lip (M)$ (for instance, $M=\mathbb{N}$). Note that every Lipschitz function on $M$ is non-local. For ($**$), let $M := ([0,1]\times \{0\}) \cup \{(0,1), (1,1)\} \subseteq \ell_1^2$. Define 
\[
f(p) = d(0',p) \quad (p\in M),
\]
where $0'=(0,0)$ is the origin in $\mathbb{R}^2$. Then $f \in \SNA(M)$ at $m_{(0,1),(1,1)}$ which is a strongly exposed point of $B_{\mathcal{F}(M)}$, however $f$ is local. 
\end{rem}

\begin{figure}[H]
\[\begin{tikzcd}
	{\overline{\SNA(M)}=\Lip (M)} && {\overline{\text{LipSE}(M)}=\Lip (M)} \\
	\\
	{B_{\mathcal{F}(M)} = \overline{\text{co}} (\text{ext} ( B_{\mathcal{F}(M)}))} && {B_{\mathcal{F}(M)} = \overline{\text{co}} (\text{str-ext} ( B_{\mathcal{F}(M)}))}
	\arrow["{(i)}"', Rightarrow, from=1-1, to=3-1]
	\arrow["{(ii)}", Rightarrow, from=1-1, dashed, to=3-3]
	\arrow[Rightarrow, from=1-3, to=1-1]
	\arrow["{(iii)}", Rightarrow, from=1-3, to=3-3]
	\arrow[Rightarrow, from=3-3, to=3-1]
\end{tikzcd}\]
\caption{
}
\label{fig:relation2}
\end{figure}

In Figure \ref{fig:relation2}, implication (i) follows from \cite[Theorem 3.3]{CGMR}. Implication (ii) holds whenever $M$ is compact and was proved in \cite[Theorem 3.15]{CGMR}. Finally, implication (iii) is an immediate consequence of the Hahn--Banach theorem.

\subsection{The metric space $\mathfrak{M}_p$}

We will observe that there exists a metric space $M$ such that $\LipSE(M)$ is dense in $\Lip(M)$, while $\mathcal F(M)$ fails all currently known sufficient conditions guaranteeing universal denseness of strongly norm-attaining Lipschitz maps. To this end, we follow the idea of \cite{CGMR} and use the metric space constructed in \cite[Theorem 2.5]{CGMR}, which was originally
introduced to show that being a universal domain metric space for $\SNA$-denseness does not imply that $\mathcal F(M)$ has the RNP.

\begin{definition}[\text{\cite[Theorem 2.5]{CGMR}}]
    Consider the subsets of $\mathbb{R}^2$ given by 
    \begin{align*}
        A_n &= \left\{ \left ( \frac{k}{2^n}, \frac{1}{2^n} : k \in \{0,\ldots, 2^n \} \right ) \right\} \subseteq \mathbb{R}^2, \quad \forall n\in\mathbb{N}, \\
        M_n &= \bigcup_{n=0}^\infty A_n, \quad M = M_\infty \cup ([0,1]\times \{0\} ). 
    \end{align*}
    Let $\mathfrak{M}_p$ be the set $M$ endowed with the distance inherited from $(\mathbb{R}^2 , \|\cdot\|_p)$ for $p=1,2$. 
\end{definition}

In \cite[Proposition 2.6]{CGMR}, it is proved that $\SNA (\mathfrak{M}_p, Y)$ is dense in $\Lip(\mathfrak{M}_p, Y)$ for every Banach space $Y$ and $p=1,2$. A closer inspection of the proof reveals that, in fact, the following stronger statement is established:
\[
\{f\in\Lip(\mathfrak{M}_p,Y): \text{$f$ is non-local} \}
\]
is dense in $\Lip(\mathfrak{M}_p,Y)$.

Therefore, Corollary~\ref{cor:non-local} immediately yields the following consequence.

\begin{prop}\label{prop:M_p_density}
    $\LipSE (\mathfrak{M}_p , Y)$ is dense in $\Lip (\mathfrak{M}_p,Y)$ for every Banach space $Y$ and $p=1,2.$
\end{prop}

\begin{example}\label{ex:M_2_plus_snowflake}
There is a complete metric space $M$ satisfying that 
\begin{itemize}
    \itemsep0.25em
    \item $\LipSE (M,Y)$ is dense in $\Lip (M,Y)$ for every Banach space $Y$; 
    \item $\mathcal{F}(M)$ fails the RNP;
    \item $\mathcal{F}(M)$ fails property quasi-$\alpha$; 
    \item $\mathcal{F}(M)$ does not contain any norming uniformly strongly exposed set.
\end{itemize}
\end{example}

Indeed, by \cite[Example 2.12]{CGMR}, the metric space 
\[
M :=\mathfrak{M}_2 \bigsqcup [0,1]^{1/2}
\]
is known to satisfy all the conditions above except the first one. To check the first condition, we need the following lemma.

\begin{lemma}\label{lem:ell_1_sum}
    Let $X,Y$ and $Z$ be Banach spaces. If $\ASE(X,Z)$ and $\ASE(Y,Z)$ are dense in $\Lin (X,Z)$ and $\Lin (Y,Z)$, respectively, then $\ASE (X\oplus_1 Y, Z)$ is dense in $\Lin (X\oplus_1 Y,Z)$.
\end{lemma}

\begin{proof}
    Let $T \in \Lin (X\oplus_1 Y,Z)$ be a norm one element and write $T = T_1P_X + T_2P_Y$ where $P_X$, $P_Y$ are the canonical projections from $X \oplus_1 Y$ into $X$ and $Y$, respectively. Note that $\|T\|=\max \{\|T_1\|,\|T_2\|\}$. Without loss of generality, we assume that $\|T_1\| =1$. Let  $\varepsilon>0$ be given. Find $(T_1', T_2') \in \ASE(X,Z) \times \ASE(Y,Z)$ such that $\|T_1'\| =1$, $\|T_2'\|= \|T_2\|$, $\|T_1-T_1' \| < \varepsilon$ and $\|T_2 - T_2' \| <\varepsilon$.  Thus the operator
$$
S(x,y) := (1+\eps)T_1'P_X(x,y) + T_2' P_Y(x,y) \in \Lin (X\oplus_1 Y,Z)
$$
satisfies that $\|S \| = 1+\varepsilon$ and $\| S-T \| <2\varepsilon$. If $((x_n,y_n))_n \subset B_{X\oplus_1 Y}$ is such that $\|S(x_n,y_n)\| \to 1+\varepsilon$, then 
    \[
    1+\varepsilon \leftarrow \|(1+\eps)T_1'(x_n) + T_2' (y_n)\| \leq (1+\varepsilon)\|x_n\| + \|T_2'\| \|y_n\| \leq  1+\varepsilon.
    \]
    Thus, $\|y_n\| \to 0$ and $\|x_n\| \to 1$; hence $(x_n)$ is convergent since $T_1' \in \ASE(X,Z)$. It follows that $(x_n, y_n)$ is convergent; and we conclude that $S \in \ASE (X \oplus_1 Y, Z)$.
\end{proof}

Now, note from Proposition \ref{prop:M_p_density} that $\ASE (\mathcal{F}(\mathfrak{M_2} ), Y)$ is dense in $\mathcal{L} (\mathcal{F}(\mathfrak{M_2} ), Y)$. Since $\mathcal{F}([0,1]^{1/2})$ has the RNP, $\ASE (\mathcal{F}([0,1]^{1/2}), Y)$ is dense in $\mathcal{L} (\mathcal{F}([0,1]^{1/2} ), Y)$ by Corollary \ref{cor:universal-domain-dense}. Therefore, Lemma \ref{lem:ell_1_sum} shows that the metric union $M= \mathfrak{M}_2 \bigsqcup [0,1]^{1/2}$ satisfies 
\[
\ASE (\mathcal{F}(M), Y) = \ASE (\mathcal{F}(\mathfrak{M}_2) \oplus_1 \F([0,1]^{1/2}),Y) 
\]
is dense in $\mathcal{L} (\mathcal{F}(M),Y)$. Indeed, it is well known that $\F(N \bigsqcup N') = \F(N) \oplus_1 \F(N')$ for every pointed metric spaces $N$ and $N'$. Consequently, Theorem \ref{theorem:Lip-linear} shows that $\LipSE (M,Y)$ is dense in $\Lip(M,Y)$. This finishes the verification of Example \ref{ex:M_2_plus_snowflake}.

\subsection{Isomorphic characterizations and residuality}\label{subsection:isomorphic}

Building on Kirchheim’s metric rectifiability theory \cite{Kirchheim}, the authors in \cite{AGPP2022} identified the absence of nontrivial curve fragments as the natural geometric condition governing locally flat Lipschitz functions. Recall that a metric space is said to be \emph{purely $1$-unrectifiable} if every Lipschitz image contained in $M$ of a subset of $\mathbb{R}$ has zero one-dimensional Hausdorff measure. They proved that for every complete metric space $M$, $\F(M)$ has the \textup{RNP} if and only if $M$ is purely $1$-unrectifiable. Using this characterization together with Theorem \ref{theorem:Lip-linear}, we obtain the following isomorphic characterization.

\begin{theorem}
Let $M$ be a complete metric space. The following are equivalent.
\begin{enumerate}
\item[\textup{(a)}] $M$ is purely $1$-unrectifiable. 
\item[\textup{(b)}] $\ASE (\mathcal{F}(N'), Y)$ is dense in $\Lin (\mathcal{F}(N'), Y)$ for each bi-Lipschitz copy $N'$ of a closed subset $N$ of $M$ and Banach space $Y$. 
\item[\textup{(c)}] $\LipSE(N',Y)$ is dense in $\Lip (N', Y)$ for each bi-Lipschitz copy $N'$ of a closed subset $N$ of $M$ and Banach space $Y$. 
\item[\textup{(d)}] $\SNA(N',Y)$ is dense in $\Lip (N', Y)$ for each bi-Lipschitz copy $N'$ of a closed subset $N$ of $M$ and Banach space $Y$. 
\end{enumerate}
\end{theorem}

\begin{proof}
    It is clear that any bi-Lipschitz copy $N'$ of any closed subset of a purely $1$-unrectifiable metric space $M$ is also purely $1$-unrectifiable. Therefore, (a)$\implies$(b) holds by the classical result of Bourgain \cite{Bourgain77}. By Theorem \ref{theorem:Lip-linear}, there is an isometric correspondence between $\ASE(\mathcal{F}(\cdot),Y)$ and $\LipSE (\cdot, Y)$. This ensures the equivalence (b)$\iff$(c).

Since (c)$\implies$(d) is immediate, it remains to prove that (d)$\implies$(a). Assume to the contrary that $M$ is not purely $1$-unrectifiable. By Kirchheim’s lemma \cite{Kirchheim} (see also \cite[Definition 2.3]{FJLPPQ}), $M$ contains the image of a bi-Lipschitz embedding $\gamma: K \to M$, where $K \subseteq \mathbb{R}$ with positive Lebesgue measure. 
It follows that $K$ is a bi-Lipschitz copy of $\gamma(K) \subseteq M$, and by \cite[Theorem 2.3]{CCGMR}, $\SNA (K)$ is not dense in $\Lip (K)$.
\end{proof}

\begin{rem}
As a consequence, we answer a question posed in \cite[p.~21]{CGMR}, where the implication (d)$\implies$(a) was left open.    
\end{rem}

Regarding residuality, Theorem \ref{theorem:Lip-linear} shows that
$\LipSE(M,Y)$ is a $G_\delta$ subset of $\Lip(M,Y)$, since
$\ASE(\F(M),Y)$ is a $G_\delta$ subset of $\mathcal L(\F(M),Y)$
(see \cite[p.~7]{JMR23}). Consequently, if $\LipSE(M,Y)$ is dense in
$\Lip(M,Y)$, then $\LipSE(M,Y)$ is residual in $\Lip(M,Y)$; in particular,
$\SNA(M,Y)$ is residual in $\Lip(M,Y)$.
Inspection of the results in \cite{JMR23}, combined with Theorem
\ref{theorem:Lip-linear}, yields the following further consequences.

\begin{prop}
Let $M$ be a complete metric space and $Y$ be a Banach space. Then, we have the following.
\begin{enumerate}[label=\textup{(\alph*)}]
\item (\cite[Theorem 4.1]{JMR23}) If $M$ and $Y^*$ are separable and $\SNA(M,Y)$ is residual, then $\LipSE(M,Y)$ is dense in $\Lip (M,Y)$.

\item (\cite[Corollary 5.1]{JMR23}) If $M$ is separable and $\F(M)$ has the Lindenstrauss property A, then $\LipSE(M)$ is dense in $\Lip (M,Y)$.

\item (\cite[Corollary 5.8]{JMR23}) If $\LipSE(M)$ is dense in $\Lip(M)$ and a Banach space $Z$ satisfies one of the following conditions: 
\begin{enumerate}
    \item $Z$ has the \textup{RNP} and $\text{str-exp}(B_Z)$ is either countable or discrete up to rotation;
    \item $Z^*$ has the \textup{RNP} and $\text{str-exp}(B_Z)$ is countable up to rotation,
\end{enumerate}
then $\LipSE(M,Z)$ is dense in $\Lip (M,Z)$.
\end{enumerate}
\end{prop}

The denseness of $\LipSE(M)$ seems to be a much stronger requirement than the denseness of $\SNA(M)$. However, we do not know whether these two denseness phenomena actually differ.

\begin{question}
Is $\overline{\LipSE(M)}=\Lip(M)$ equivalent to $\overline{\SNA(M)} = \Lip(M)$?
\end{question}

\section{Weak sequential denseness of strongly exposing Lipschitz maps}\label{section:weak}

\subsection{Positive results}
As mentioned already, the set $\SNA(M)$ is weakly sequentially dense
in $\Lip(M)$ for every metric space $M$ (\cite[Theorem 4.1]{CCGMR}). Therefore, whenever strongly norm-attaining Lipschitz maps on $M$ can be approximated (either in norm or in the weak topology) by elements of $\LipSE(M)$, it follows that $\LipSE(M)$ is also weakly sequentially dense. The following result shows that this strategy indeed works when $M$ is a compact metric space containing no isometric copy of $[0,1]$.

\begin{prop}\label{prop:isometric_copy_[0,1]}
    Let $M$ be a compact metric space which does not contain any isometric copy of $[0,1]$. Then $\LipSE (M)$ is weakly sequentially dense in $\Lip (M)$.
\end{prop}

\begin{proof}
    Let $f \in \Lip (M)$ be given. From weak sequential denseness of $\SNA(M)$, take a sequence $(g_n) \subset \SNA (M)$ such that $g_n \xrightarrow{w} f$. Since $M$ does not contain any isometric copy of $[0,1]$, we can find $m_{p_n,q_n} \in \ext (B_{\mathcal{F}(M)})$ such that $g_n$ attains its norm at $(p_n, q_n)$ for each $n \in \mathbb{N}$ \cite[Lemma 3.1]{CGMR}. Since $M$ is compact, by \cite[Lemma 3.12]{CGMR}, there exist non-local Lipschitz maps $h_n \in \Lip (M,Y)$ such that $\|g_n-h_n \| \to 0$. Finally, Corollary \ref{cor:non-local} implies that there exists $f_n \in \LipSE (M,Y)$ such that $\|h_n- f_n\| \to 0$; hence $f_n \xrightarrow{w} f$.
\end{proof}

\begin{example}
Let $M$ be a nowhere dense closed subset of $[0,1]$ whose Lebesgue measure is positive. Then $\SNA (M)$ is not norm dense in $\Lip (M)$ (see \cite[Theorem 2.3]{CCGMR}). In particular, $\LipSE(M)$ is not norm dense in $\Lip (M)$. While Proposition \ref{prop:isometric_copy_[0,1]} shows that $\LipSE(M)$ is weakly sequentially dense in $\Lip(M)$.
\end{example}

The following results generalizes the proof in \cite[Theorem 4.1]{CCGMR} and \cite[Theorem 2.6]{KMS} for more general metric space. However, note that this result does not ensure that the norm of approximating Lipschitz functions converges to the norm of the target function.

\begin{theorem}\label{theorem:peak-ball-weak-dense}
Let $M$ be a complete metric space. Suppose that there exist a sequence of distinct pairs $((p_n,q_n))_n \subseteq M \times M$, a sequence $(r_n) \subseteq [0,1]$, and another sequence $(\eps_n) \subseteq (0,1)$ such that
\begin{enumerate}
\item[\textup{(i)}] $m_{p_n,q_n} \in \text{str-exp}(B_{\free(M)})$ for each $n \in \N$,
\item[\textup{(ii)}] $\sup_n \beta_n < 1$ for $\beta_n = \max \bigl\{ \frac{r_n}{r_n+\eps_n}, \frac{1-r_n}{(1-r_n)+\eps_n} \bigr\}$,
\item[\textup{(iii)}] $U_n$ and $U_m$ are mutually disjoint for all $n\neq m$ for the set
$$
U_n := B(p_n,(r_n+\eps_n)d(p_n,q_n)) \cup B(q_n,[(1-r_n)+\eps_n]d(p_n,q_n)).
$$
\end{enumerate}
Then, $\LipSE(M)$ is weakly sequentially dense in $\Lip(M)$.
\end{theorem}

\begin{proof}
We follow the proof of \cite[Theorem 2.6]{KMS}. Let $g \in \Lip(M)$ with $\|g\|=1$. We claim that there exists a sequence $(f_n)$ in $\LipSE(M)$ converging weakly to $g$. 

Write $M_n := (M \setminus U_n) \cup \{p_n,q_n\}$, and define a function $h_n$ on $M_n$ for each $n \in \mathbb{N}$ by
$$
h_n(x) = \begin{cases}
\, g(x) \,, & \quad x \in M_n \setminus \{p_n,q_n\} \\
\, r_ng(p_n) + (1-r_n)g(q_n) \\
+ \sgn[g(p_n)-g(q_n)]r_n\bigl(1+\frac{1}{1-\beta_n}\bigr) d(p_n,q_n) \,, & \quad x= p_n \\
\, r_ng(p_n) + (1-r_n)g(q_n) \\
- \sgn[g(p_n)-g(q_n)](1-r_n)\bigl(1+\frac{1}{1-\beta_n}\bigr) d(p_n,q_n) \,, & \quad x= q_n
\end{cases}
$$
where $\beta_n = \max \bigl\{ \frac{r_n}{r_n+\eps_n}, \frac{1-r_n}{(1-r_n)+\eps_n} \bigr\}<1$ for each $n \in \mathbb{N}$. Then, if $p,q \in M_n \setminus \{p_n,q_n\}$, then clearly
$$
\frac{|h_n(p) - h_n(q)|}{d(p,q)} \leq 1,
$$
and
$$
\frac{|h_n(p_n)-h_n(q_n)|}{d(p_n,q_n)} = \frac{\bigl(1+\frac{1}{1-\beta_n}\bigr)d(p_n,q_n)}{d(p_n,q_n)} = 1 + \frac{1}{1-\beta_n}.
$$
If $q \in M_n \setminus \{p_n,q_n\}$, then
\begin{align*}
&\frac{|h_n(p_n)-h_n(q)|}{d(p_n,q)} \\
&= \frac{r_ng(p_n) + (1-r_n)g(q_n) + \sgn[g(p_n)-g(q_n)]r_n\bigl(1+\frac{1}{1-\beta_n}\bigr)d(p_n,q_n)-h_n(q)}{d(p_n,q)} \\
&\leq \frac{|g(p_n)-g(q)|}{d(p_n,q)}   \\
 &\quad +\frac{|(1-r_n)[g(q_n)-g(p_n)] + \sgn[g(p_n)-g(q_n)]r_n\bigl(1+\frac{1}{1-\beta_n}\bigr)d(p_n,q_n)|}{d(p_n,q)} \\
 &=: \frac{|g(p_n)-g(q)|}{d(p_n,q)}  + (\dagger).
\end{align*}
Note that 
\begin{align}
    \Big|(1-r_n)[g(q_n)-g(p_n)] &+ \sgn[g(p_n)-g(q_n)]r_n\!\left(1+\frac{1}{1-\beta_n}\right)d(p_n,q_n) \Big| \nonumber \\
    &= \left| (1-r_n) |g(q_n)-g(p_n)| - r_n \!\left( 1+\frac{1}{1-\beta_n}\right) d(p_n,q_n) \right| \label{eq:1-r_n...(1)}
\end{align}
Since $\varepsilon_n \leq r_n$, we have
\[
\beta_n \geq \frac{1-r_n}{(1-r_n)+\eps_n} \geq 1-r_n,
\]
which implies that 
\begin{equation}\label{eq:r_n(1+1/1-beta_n)...(1)}
    r_n\!\left(1+\frac{1}{1-\beta_n}\right) \geq r_n\!\left(1+\frac{1}{r_n}\right) \geq 1 \geq 1-r_n.
\end{equation}
Hence, the quantity inside the absolute value on the right-hand side of \eqref{eq:1-r_n...(1)} is less than or equal to zero.
It follows that
\begin{align*}
    (\dagger) &\stackrel{(\ref{eq:1-r_n...(1)})}{=} \frac{\left| (1-r_n) |g(q_n)-g(p_n)| - r_n \left( 1+\frac{1}{1-\beta_n}\right) d(p_n,q_n) \right|}{d(p_n,q)} \\
    &= \frac{ r_n \left( 1+\frac{1}{1-\beta_n}\right) d(p_n,q_n) -(1-r_n) |g(q_n)-g(p_n)|  }{d(p_n,q)} \leq \frac{ r_n \left( 1+\frac{1}{1-\beta_n}\right) d(p_n,q_n)}{d(p_n,q)} 
\end{align*}

Finally, using $d(p_n,q) \geq (r_n+\varepsilon_n) d(p_n,q_n)$, we obtain that 
{
\begin{align*}
\frac{|g(p_n)-g(q)|}{d(p_n,q)}  + (\dagger) &\leq 1+ \frac{ r_n \bigl(1+\frac{1}{1-\beta_n}\bigr)d(p_n,q_n)}{d(p_n,q)} \\
&\leq 1+ \frac{r_n}{r_n+\eps_n} \left(1 + \frac{1}{1-\beta_n}\right)  \\
&\leq 1 + \beta_n \left(1 + \frac{1}{1-\beta_n}\right) < 1 + \frac{1}{1-\beta_n}.
\end{align*}}
Similarly, if $p \in M_n \setminus \{p_n,q_n\}$, then
$$
\frac{|h_n(p)-h_n(q_n)|}{d(p,q_n)} \leq 1+ \beta_n \left(1 + \frac{1}{1-\beta_n}\right) < 1+ \frac{1}{1-\beta_n}.
$$
Therefore, $h_n \in \SNA(M_n)$ at $m_{p_n,q_n}$ for every $n \in \mathbb{N}$. By Mcshane extension, we may extend each $h_n$ to $g_n$ on $M$ for each $n \in \N$. Note that $\supp (g_n-g)$ is contained in $U_n$ for every $n\in\mathbb{N}$. In particular, $\supp (g_n-g) \cap \supp (g_m-g)=\emptyset$ for $n\neq m$. Thus, by \cite[Lemma 1.5]{CCGMR}, $(g_n)$ converges weakly to $g$.  

Since $g_n \in \SNA(M)$ at $m_{p_n,q_n}$ and $m_{p_n,q_n}$ is assumed to be a strongly exposed point of $B_{\mathcal{F}(M)}$, each $g_n$ belongs to the norm closure of $\LipSE (M)$ by Lemma \ref{lemma:LipSE-ASE}. It follows that there exists $(f_n)$ in $\LipSE(M)$ such that $\|f_n-g_n\| \to0$. Thus, $(f_n)$ converges weakly to $g$. 
\end{proof}

The following result shows that the assumptions in Theorem \ref{theorem:peak-ball-weak-dense} are weaker than those in \cite[Lemma 4.2]{CCGMR} together with the condition $m_{p_n,q_n} \in \text{str-exp}(B_{\F(M)})$, and moreover, even when $\frac{d(p_n,q_n)}{r_n} \to 0$ (which was required only for the norm convergence of $\|f_n\|$)  is replaced by $\sup_n \frac{d(p_n,q_n)}{r_n} < 1$.

\begin{cor}\label{cor:disjoint-ball-p_n}
Let $M$ be a complete metric space. Suppose that there exists a sequence $B(p_n, r_n)$ of disjoint balls of $M$ and a sequence $(q_n)$ in $M$ such that $0<d(p_n,q_n) <r_n$ with $\sup_n \frac{d(p_n,q_n)}{r_n} < 1$. If $m_{p_n,q_n} \in \str (B_{\mathcal{F}(M)})$ for each $n \in \mathbb{N}$, then $\LipSE(M)$ is weakly sequentially dense in $\Lip(M)$.
\end{cor}

\begin{proof}
    We will show that the assumptions imply those in Theorem \ref{theorem:peak-ball-weak-dense}. Given such sequences $(p_n)$, $(q_n)$ in $M$ and $(r_n)$, write $\alpha := \sup_n \frac{d(p_n,q_n)}{r_n} <1$. Consider the sequences $p_n' = p_n$, $q_n' = q_n$, $r_n' = \frac{1 + \alpha}{2}$ and $\eps_n' = \frac{1 - \alpha}{4}$ for each $n \in \N$. It is clear that
$$
\sup_n \beta_n = \max \left\{ \dfrac{r_n'}{r_n'+\eps_n'} , \dfrac{1-r_n'}{(1-r_n')+\eps_n'} \right\} = \dfrac{\,\,\frac{1+\alpha}{2}\,\,}{\,\,\frac{3+\alpha}{4}\,\,} <1.
$$
Moreover, $U_n$ and $U_m$ are mutually disjoint for all $n \neq m$ since $U_n \subseteq B(p_n,r_n)$ for all $n \in \N$ by construction. Indeed, we have $B(p_n',\bigl( \frac{3+\alpha}{4}\bigr)d(p_n',q_n')) \subseteq B(p_n,r_n)$ and also that $B(q_n',\bigl(\frac{3-3\alpha}{4}\bigr)d(p_n',q_n')) \subseteq B(p_n,r_n)$ since
$$
d(p_n',q_n') + \frac{3-3\alpha}{4} \cdot d(p_n',q_n') \leq \dfrac{7-3\alpha}{4} \cdot \alpha r_n < r_n.
$$
This proves the claim as $((p_n',q_n'))$, $(r_n')$ and $(\eps_n')$ are the desired sequences in Theorem \ref{theorem:peak-ball-weak-dense}.
\end{proof}

Let us highlight with the following example that the assumptions in Theorem \ref{theorem:peak-ball-weak-dense} are not equivalent to those in Corollary \ref{cor:disjoint-ball-p_n} in general.

\begin{example}
Let $M = \{(u,v) \in [0,1]^2 : uv = 0\} \subseteq \mathbb{R}^2$. Then, there does not exist a sequence $(B(p_n,r_n))_n$ of disjoint balls of $M$ and a sequence $(q_n)$ of points of $M$ as in Corollary \ref{cor:disjoint-ball-p_n}. Indeed, in order to make $m_{p_n,q_n}$ a strongly exposed point, $p_n$ and $q_n$ must be in the different segment among $[0,1] \times \{0\}$ and $\{0\} \times [0,1]$. However, it follows that $B(p_n,q_n)$ must contain $0$, so that the remaining pairs $p_n$ and $q_n$ lie in the segment, which is a contradiction. \\
On the other hand, it is routine to see that Theorem \ref{theorem:peak-ball-weak-dense} can be applied to $M$ in contrast. So, $\LipSE(M)$ is weakly sequentially dense in $\Lip(M)$.
\end{example}

\subsection{Proof of Theorem \ref{theorem:derived_set}}
Before proving Theorem \ref{theorem:derived_set}, we establish the following \emph{closer pair lemma}, which will be used repeatedly.

\begin{lemma}[Closer Pair Lemma]\label{lem:making_points_closer}
    Let $M$ be a complete metric space. If $m_{u,v} \not\in \str (B_{\mathcal{F}(M)})$, then there exists $z_1 \in M\setminus \{u,v\}$ and a pair $(p_1,q_1)$ among $(u,z_1)$ and $(z_1, v)$ such that 
    \begin{equation}\label{eq:making_points_closer}
    \max \{d(z_1, u), d(z_1, v)\} \leq d(u,v) \quad \text{ and } \quad d(p_1,q_1) \leq \frac{3}{4}d(u,v).
    \end{equation}
\end{lemma}

\begin{proof}
Since $(u,v)$ has the property (Z) (see \eqref{iff_strexp_Z}), there exists $z_1 \in M \setminus \{u,v\}$ such that 
\begin{equation}\label{eq:Gromov_condition_1}
    d(u,z_1) + d(z_1,v) - d(u,v) < \frac{1}{2} \min\{d(u,z_1), d(z_1,v) \}.
\end{equation}
In particular, \eqref{eq:Gromov_condition_1} implies that 
\begin{equation}\label{eq:Gromov_conseqeunces_1}
    d(u,z_1) \leq d(u,v) \quad \text{ and } \quad  d(z_1,v) \leq d(u,v).
\end{equation}
Moreover, \eqref{eq:Gromov_condition_1} implies that either $d(u,z_1)$ or $d(z_1,v)$ is smaller than or equal to $\frac{3}{4} d(u,v)$. Choose a pair $(p_1,q_1)$ among $(u,z_1)$ and $(z_1, v)$ such that $d(p_1, q_1) \leq \frac{3}{4} d(u,v)$. 
\end{proof}

\begin{proof}[Proof of Theorem \ref{theorem:derived_set}]

\textit{Case 1}. $M'$ is empty. In this case, we distinguish two subcases.

\textit{Case 1-1.} $M$ is uniformly discrete. 

Then $\mathcal{F}(M)$ has the RNP; so $\LipSE (M)$ is norm dense in $\Lip (M)$.

\textit{Case 1-2}. $M$ is discrete but not uniformly discrete. 

Using Lemma \cite[Lemma 4.3]{CCGMR}, given $k\geq 2$ and $\varepsilon>0$, find $u,v \in M$ such that $0<d(u,v) \leq \varepsilon$ and 
\begin{equation}\label{eq:outside_ball_non_uniformly_discrete}
    \text{$M\setminus B(u, (k+1)d(u,v))$ is not uniformly discrete.}
\end{equation}

\textbf{Claim:} There exist $p,q \in M$ such that 
\begin{enumerate}[label=(\roman*)]
\itemsep0.25em
    \item $m_{p,q} \in \str (B_{\mathcal{F}(M)})$; \label{item:pq_str} 
    \item $d(p,q) \leq d(u,v)$ and $d(p,u) \leq d(u,v)$; \label{item:pq_less_than_uv}
    \item $M \setminus B(p, k d(p,q))$ is not uniformly discrete. \label{item:not_unif.discrete}
\end{enumerate}

If $m_{u,v}$ belongs to $\str (B_{\mathcal{F}(M)})$, then taking $(p,q) =(u,v)$ proves the claim. Suppose otherwise. Then by Lemma \ref{lem:making_points_closer}, there exists $z \in M \setminus \{u,v\}$ and a pair $(p_1,q_1)$ among $(u,z_1)$ and $(z_1, v)$ satisfying \eqref{eq:making_points_closer}.

If $m_{p_1,q_1} \in \str (B_{\mathcal{F}(M)})$, then $(p,q)=(p_1,q_1)$ satisfies the conditions \ref{item:pq_str}, \ref{item:pq_less_than_uv} in the claim.

If not, i.e., $m_{p_1,q_1} \notin \str (B_{\mathcal{F}(M)})$ we repeat the same argument. Thus, at the $n$-th step, either we obtain a pair $(p_n,q_n)$ such that 
\[
m_{p_n,q_n} \in \str (B_{\mathcal{F}(M)}) \quad \text { and } \quad d(p_n,q_n)\leq d(u,v),
\]
or we produce a new pair $(p_{n+1},q_{n+1})$ such that $d(p_{n+1},q_{n+1}) \leq \frac{3}{4}d(p_n,q_n)$.

If the process does not terminate, then we obtain a sequence $(p_n,q_n)$ of pairs of points in $M$. Since each new pair shares one endpoint with the preceding pair, we may choose a sequence $(w_n)$ with $w_n \in \{p_n,q_n\}$ such that 
\[
d(w_n, w_{n+1}) \leq d(p_n,q_n) \leq \left(\frac{3}{4}\right)^{\!n} \! d(u,v),
\]
and so $(w_n)$ is Cauchy. Hence the sequence $(w_n)$, which contains infinitely many distinct points, converges to some point of $M$. This point is then a cluster point of $M$, contradicting the assumption that $M'$ is empty. Therefore, the process must terminate after finitely many steps, yielding the pair $(p,q)$ satisfying \ref{item:pq_str} and  \ref{item:pq_less_than_uv}. Condition \ref{item:not_unif.discrete} follows directly from \eqref{eq:outside_ball_non_uniformly_discrete} once we verify the inclusion
\[
B(p, k d(p,q)) \subseteq B(u, (k+1) d(u,v)).
\]
Indeed, if $x \in B(p, k d(p,q))$, then 
\[
d(x,u) \leq d(x,p) + d(p,u) \leq kd(p,q) + d(u,v) \leq (k+1)d(u,v);
\]
so $x \in B( u, (k+1) d(u,v))$ and the claim has been proved. 

By applying the claim inductively, we find sequences $(p_n)$, $(q_n)$ in $M$ such that for every $n \in \mathbb{N}$, we have $m_{p_n,q_n} \in \text{str-exp}(B_{\F(M)})$, the set $M \setminus \bigcup_{m=1}^n B(p_m , 2m d(p_m,q_m))$ is discrete but not uniformly discrete, $p_{n+1}, q_{n+1} \in M \setminus \bigcup_{m=1}^n B(p_m, 2md(p_m,q_m))$, and that 
\[
d(p_{n+1}, q_{n+1} ) \leq \min \left\{ \frac{n}{n+1} d(p_n,q_n), n^{-2} \right\}.
\]
It is routine to check that the balls $\{B(p_n,n d(p_n, q_n) )\} $ are pairwise disjoint and satisfy the assumption of Corollary \ref{cor:disjoint-ball-p_n}. This completes the Case 1-2.
\medskip

\textit{Case 2}. $M'$ is non-empty and finite. Suppose that $M' = \{a_1,\ldots, a_k\}$. Moreover, we may assume that $a_1 =0$. Given $\varepsilon>0$, denote by $E_\varepsilon:= \bigcup_{i=1}^k [M \setminus B(a_i,\varepsilon)]$. 

\textit{Case 2-1}. $E_\varepsilon$ is finite for every $\varepsilon>0$. In this case, $M$ is compact and countable. So, $\mathcal{F}(M)$ has the RNP. 

\textit{Case 2-2}. Thus, we may assume that there exists $0<\varepsilon_0 < \frac{1}{4} \min_{i\neq j} d(a_i, a_j)$ such that $E_{\varepsilon_0}$ is infinite. Note that $E_\varepsilon$ is discrete for every $\varepsilon>0$. If there is $0<\varepsilon\leq \varepsilon_0$ such that $E_\varepsilon$ is not uniformly discrete in $M$, then the same argument as in Case 1-2 yields a sequence of pairwise disjoint balls to which Corollary \ref{cor:disjoint-ball-p_n} applies. Thus, we may also assume that $E_\varepsilon$ is infinite and uniformly discrete in $M$ for every $0<\varepsilon\leq \varepsilon_0$. By rescaling the metric space, we may assume that $\varepsilon_0 = 2^{-1}$. For $n \in \mathbb{N}$ and $i \in \{1,\ldots, k\}$, let us denote $C_n^i := E_{(n+1)^{-1}} \cap B(a_i, n^{-1})$ and 
\[
\alpha_n^i := \inf \{ d(x, M\setminus \{x\} ) : x \in C_n^i \}, 
\]
with the convention that $\inf \emptyset = +\infty$. 

If $\liminf_{n\to\infty} n\alpha_n^i>0$ for every $i\in\{1,\ldots,k\}$, then the proof of \cite[Theorem 4.1]{CCGMR} shows that $\mathcal F(M)$ has the RNP. Thus, we assume that there exists $i \in \{1,\ldots, k\}$ such that $\liminf_{n\to\infty} n \alpha_n^i = 0$. 
\medskip

\emph{Step 1}. We claim that there exist sequence $(j_n)$ in $\mathbb{N}$ and sequences $(p_n)$, $(q_n)$ in $M$ such that 
\begin{enumerate}[label=(\alph*)]
    \itemsep0.25em 
    \item $3n d(p_n,q_n) < \left ( \frac{3n-1}{3n} \right) (j_n + 1)^{-1}  $ for every $n$;\label{item:condition_dp_nq_n}
    \item $4 j_{n+1}^{-1} < \left ( \frac{3n-1}{3n} \right)  (j_n +1 )^{-1}   - 3n d(p_n,q_n) $  for every $n$; \label{item:condition_j_n+1}
    \item $p_n \in C_{j_n}^i$ for every $n$. \label{item:condition_p_n}
\end{enumerate}

Indeed, notice first that $\liminf_{n\to\infty} \left( \frac{18n^2}{3n-1}\right) \alpha_n^i=0.$ Take $j_1 \geq 3$ such that $6j_1 \alpha_{j_1}^i < 1$. Then there exists $p_1 \in C_{j_1}^i$ such that 
\[
3 d(p_1, M\setminus \{p_1\}) < 2^{-1} j_1^{-1} \leq \frac{2}{3}\, (j_1 + 1)^{-1} .
\]
Thus, there is $q_1 \in M$ such that $3 d(p_1, q_1) < \frac{2}{3} (j_1 + 1)^{-1} .$ Now, assume that we have defined $p_n$, $q_n$ and $j_n$, and let us define $p_{n+1}$, $q_{n+1}$ and $j_{n+1}$. By condition \ref{item:condition_dp_nq_n}, we can take $j_{n+1}\in\mathbb{N}$ such that 
\[
4j_{n+1}^{-1} < \left ( \frac{3n-1}{3n} \right) (j_n + 1)^{-1}  - 3n d(p_n, q_n) \,\,\,  \text{and} \,\,\, \left(\frac{18(n+1)^2}{3(n+1)-1} \right) j_{n+1} \alpha_{j_{n+1}}^i < 1.
\]
Then there are $p_{n+1} \in C_{j_{n+1}}^i $ and $q_{n+1} \in M$ such that 
\[
3(n+1) d(p_{n+1},q_{n+1}) < \left(\frac{3(n+1)-1}{6(n+1)} \right) j_{n+1}^{-1} \leq \left ( \frac{3(n+1)-1}{3(n+1)} \right)(j_{n+1} + 1)^{-1} .
\]
This completes the construction of the sequences $(p_n)$, $(q_n)$ and $(j_n)$. 
\medskip

\emph{Step 2}. We claim that there exist $(u_n)$, $(v_n)$ in $M$ such that 
\begin{enumerate}[label=(\alph*),start=4]
    \itemsep0.25em 
    \item $d(u_n,p_n) \leq d(p_n, q_n)$ for every $n$;\label{item:du_np_n}
    \item $m_{u_n, v_n} \in \str (B_{\mathcal{F}(M)})$ for every $n$. \label{item:mu_nv_n}
\end{enumerate}
To prove the claim, fix $n \in \mathbb{N}$. If $m_{p_n,q_n} \in \str (B_{\mathcal{F}(M)})$, then $(u_n,v_n)=(p_n,q_n)$ works. If not, apply Lemma \ref{lem:making_points_closer} so that we  find $z_1 \in M \setminus \{p_n,q_n\}$ and a pair $(r_1,s_1)$ among $(p_n,z_1)$ and $(z_1, q_n)$ satisfying \eqref{eq:making_points_closer}. 

 Note that 
\[
\max \{ d(p_n,z_1), d(z_1, q_n)\}  \leq d(p_n,q_n) < \left( \frac{3n-1}{9n^2} \right)(j_n + 1 )^{-1},
\]
where the last inequality follows from \ref{item:condition_dp_nq_n}. Thus, 
\begin{equation}\label{eq:upper_bound}
d(r_1, a_i) \leq d(r_1, p_n) + d(p_n, a_i) < \left( \frac{3n-1}{9n^2} \right)(j_n + 1 )^{-1} + j_n^{-1}
\end{equation}
and 
\begin{equation}\label{eq:lower_bound}
d(r_1,a_i) \geq d(p_n,a_i) - d(r_1, p_n) > (j_{n}+1)^{-1} -  \left( \frac{3n-1}{9n^2} \right)(j_n + 1 )^{-1}
\end{equation}
since $p_n \in C_{j_n}^i$. 
If $m_{r_1,s_1} \in \str(B_{\mathcal{F}(M)})$, then the pair $(u_n, v_n)=(r_1,s_1)$ satisfies \ref{item:du_np_n} and \ref{item:mu_nv_n}. If not, we repeat the same argument. Thus, at the $k$-th step, either we obtain a pair $(r_k, s_k)$ satisfying \ref{item:du_np_n} and \ref{item:mu_nv_n} or we produce a new pair $(r_{k+1}, s_{k+1})$ such that $d(r_{k+1}, s_{k+1})\leq \frac{3}{4} d(r_k,s_k)$ and the estimates \eqref{eq:upper_bound} and
\eqref{eq:lower_bound} remain valid after replacing $r_1$ by $r_{k+1}$.

If the process does not terminate, then we obtain a Cauchy sequence $(w_k)$, which contains infinitely many different points, with $w_k \in \{ r_k, s_k\}$. Since $d(w_k,r_k)\leq d(r_k,s_k) \to 0$ as $k\to\infty$, the sequences $(w_k)$ and $(r_k)$ have the same limit, say $w_\infty$. As each $r_k$ satisfies \eqref{eq:upper_bound} and
\eqref{eq:lower_bound}, we obtain that 
\begin{align}\label{eq:dw_infty_a_i}
0< \left( 1- \frac{3n-1}{9n^2} \right)(j_n + 1 )^{-1} \leq d(w_\infty, a_i) \leq \left( \frac{3n-1}{9n^2} \right)(j_n + 1 )^{-1} + j_n^{-1} \overset{(\star)}{\leq} \frac{7}{18} < \varepsilon_0,    
\end{align}
where $(\star)$ holds since $j_n \geq 3$, so 
\[
\left( \frac{3n-1}{9n^2} \right)(j_n + 1 )^{-1} + j_n^{-1}  \leq \left( \frac{2}{9} \right) \frac{1}{4} + \frac{1}{3} = \frac{7}{18}.
\]

On the other hand, being a cluster point of $M$, $w_\infty$ must belong to $M'=\{a_1,\ldots,a_k\}$, which contradicts \eqref{eq:dw_infty_a_i}.

It follows that the process must terminate after finitely many steps, so we obtain the pair $(u_n, v_n)$ that satisfies \ref{item:du_np_n} and \ref{item:mu_nv_n}. This completes the proof of Step 2. Notice that \eqref{eq:upper_bound} and \eqref{eq:lower_bound} hold with $r_1$ replaced by $u_n$, that is, 
\begin{equation}\label{eq:upper_lower_bound2}
\left( 1- \frac{3n-1}{9n^2} \right)(j_n + 1 )^{-1}  < d(u_n, a_i)  < \left( \frac{3n-1}{9n^2} \right)(j_n + 1 )^{-1} + j_n^{-1}
\end{equation}
\medskip

\emph{Step 3}. We claim that 
\begin{equation}\label{eq:disjoint_balls_un_ai}
B(u_n, 3nd(p_n,q_n)) \cap B (a_i, 4j_{n+1}^{-1} ) = \emptyset.    
\end{equation}
Indeed, if the set is non-empty, then this implies that 
\[
d(u_n, a_i) \leq 3nd(p_n,q_n) +  4j_{n+1}^{-1} < \left ( 1-\frac{1}{3n} \right)  (j_n +1 )^{-1}
\]
where the second inequality follows from \ref{item:condition_j_n+1}. This contradicts \eqref{eq:upper_lower_bound2} since 
\[
d(u_n, a_i) > \left( 1- \frac{3n-1}{9n^2} \right)(j_n + 1 )^{-1} = \left( 1- \frac{1}{3n} +\frac{1}{9n^2} \right)(j_n + 1 )^{-1}.
\]
Moreover, if $m>n$, then 
\begin{align*}
    &B(u_m , 3m d(p_m, q_m) ) \\
    &\quad\subseteq B\left( u_m, \left( \frac{3m-1}{3m} \right) (j_m +1 )^{-1} \right) \\ 
    &\quad\subseteq B\left(a_i, \left( \frac{3m-1}{3m} \right) (j_m +1 )^{-1} + \left( \frac{3m-1}{9m^2} \right)(j_m + 1 )^{-1} + j_m^{-1} \right) \\
    &\quad\subseteq B\left (a_i, \left( 2-\frac{1}{9m^2} \right)j_m^{-1} \right) \subseteq B(a_i, 2 j_{n+1}^{-1} ) \subseteq B(a_i, 4 j_{n+1}^{-1} ).
\end{align*}
This, combined with \eqref{eq:disjoint_balls_un_ai}, shows that 
\[
B(u_m, 3m d(p_m, q_m) )  \cap B(u_n, 3nd(p_n,q_n)) = \emptyset
\]
whenever $n<m$. Therefore, we can apply Corollary \ref{cor:disjoint-ball-p_n} to obtain that $\LipSE (M)$ is weakly sequentially dense. This completes Case 2-2.
 \end{proof}

\subsection{Negative results}

We have already observed that $\LipSE(M)$ may be trivial (for instance, when $M$ is a length space). The following results show that being non-length is not sufficient to guarantee that $\LipSE(M)$ is weakly sequentially dense in $\Lip(M)$.

\begin{prop}
Let $N$ be a metric space and $E$ a finite metric space. Let $M = N \bigsqcup E$ be the metric union of $N$ and $E$. If  $\LipSE(M)$ is weakly sequentially dense in $\Lip(M)$, then $\LipSE(N)$ is weakly sequentially dense in $\Lip(N)$.
\end{prop}

\begin{proof}
Fix a non-zero $h \in \Lip (N)$. Then $f := (h,0)$ can be approximated weakly by a sequence $(g_n) \subset \LipSE (M)$. By Theorem \ref{theorem:Lip-linear} (and the observation in the paragraph preceding subsection \ref{subsection:isomorphic}), we have 
$$
T_{g_n} \in \text{SE}(\F(N) \oplus_1 \F(E)) \subseteq \F(N)^* \oplus_\infty \F(E)^*.
$$ 
Write $T_{g_n} = (x_n^*,y_n^*)$, where $x_n^* \in \mathcal{F}(N)^*$ and $y_n^* \in \mathcal{F}(E)^*$ for each $n\in\mathbb{N}$.  Since $\mathcal{F}(E)$ is finite-dimensional and $y_n^* \xrightarrow{w} 0$, the sequence $(\|y_n^*\|)$ converges to $0$. Moreover, since $(x_n^*)$ is weakly bounded, passing to a subsequence, we assume that $\|x_n^*\| \to \beta$ for some $\beta>0$. By passing to a subsequence if necessary, we assume that $\|x_n^*\| >\beta/2$ and $\|y_{n}^*\| < \beta/2$ for all $n\in\mathbb{N}$. This implies that $\|g_n\| = \|x_n^*\|$ for every $n\in\mathbb{N}$. 

Let $n \in \mathbb{N}$ be fixed and $m_{p_n,q_n}$ be the point strongly exposed by $T_{g_n}$. If $(\mu_k) \subset B_{\mathcal{F}(N)}$ satisfies that $x_n^* (\mu_k) \to \|x_n^*\|$ as $k \to \infty$, then 
\[
\lim_{k\to\infty}T_{g_n} (\mu_k,0)= \lim_{k\to\infty} x_n^*(\mu_k) \to \|x_n^*\| =\|g_n\|.
\]
Thus, we can find a subsequence $(\mu_{j_k})$ and $\theta \in \{-1,1\}$ such that $(\mu_{j_k},0) \to \theta m_{p_n,q_n}$. It follows that $x_n^* \in \SE ( \mathcal{F}(N))$. Since $x_n^* \xrightarrow{w} h$ and $h$ was arbitrary, the proof is complete.
\end{proof}

\begin{cor}\label{cor:length-finite}
    Let $M=L \bigsqcup E$ be the metric union of $L$ and $E$, where $L$ is a length metric space and $E$ is a finite metric space. Then $\LipSE (M)$ is not weakly sequentially dense in $\Lip(M)$.
\end{cor}

For the next result, let us comment that every proper length space (in particular, compact length space) is geodesic.

\begin{theorem}\label{theorem:compact-proper-length}
    Let 
    $$
    M = L_1 \cup \cdots \cup L_n
    $$ be a complete metric space, where $L_1, \ldots, L_{n-1}$ are compact length, $L_n$ is proper length, and $\dist (L_j, L_k) >0$ for every $j \neq k$. Then, $\LipSE (M)$ is not weakly sequentially dense in $\Lip(M)$.
\end{theorem}

\begin{proof}
    Fix any $x_j \in L_j$ for each $j = 1, \ldots, n$, and let $\delta :=  \min_{j,k} d(L_j,L_k) > 0$. Define a Lipschitz function $f \in \Lip(M)$ by
    $$
    f(p) = \sum_{j=1}^n \max \left\{ \frac{\delta}{2} - d(p,x_j), 0\right\}.
    $$
    Then, $\|f\| = 1$ since each $L_j$ is length. By the construction of $f$, we have
    \begin{equation}\label{eq:disjoint_segment_small_values}
    |f(m_{u,v})| \leq \frac{1}{2}    
    \end{equation}
    for any $u \in L_j$ and $v \in L_k$ with $j \neq k$. Assume now that a sequence $(f_k) \subseteq \LipSE(M)$ converges weakly to $f$. Passing by a subsequence if necessary, there exists $\beta \geq 1$ such that $ \lim_k \|f_k\| = \beta \geq 1$. As each $f_k$ is strongly exposing, it must attain its norm at some $m_{p_k,q_k} \in \str(B_{\F(M)})$ where $p_k$ and $q_k$ do not lie in the same $L_j$. Without loss of generality, we will only consider the case when $p_k \in L_1$ and $q_k \in L_{n-1}$, or when $p_k \in L_1$ and $q_k \in L_n$ for all $k \in \N$ passing by a subsequence.
    \medskip

    \textit{Case 1}: $p_k \in L_1$ and $q_k \in L_{n-1}$ for each $k \in \N$. \\
    If $n=2$, then this case will be covered by Case 2. We may therefore assume that $n>2$, so $d(L_1, L_{n-1})>0$.
    Since $L_1$ and $L_{n-1}$ are compact, $(f_k)$ converges uniformly to $f$ on $L_1 \cup L_{n-1}$. That is, 
$$
\sup_{u \in L_1,\,v \in L_{n-1}} |f_k(m_{u,v})-f(m_{u,v})| \to 0
$$
as $k\to\infty.$ It follows that 
\[
|\|f_k\| - f(m_{p_k,q_k})| = |f_k (m_{p_k,q_k}) - f(m_{p_k,q_k})| \to 0, 
\]
which implies that 
\[
\lim_{k\to\infty} |f(m_{p_k,q_k})| = \beta.
\]
However, we have from \eqref{eq:disjoint_segment_small_values} that $|f(m_{p_k,q_k})| \leq 1/2$ for all $k \in \mathbb{N}$; so $\beta \leq 1/2$. This contradicts $\beta \geq 1$.
    \medskip

    \textit{Case 2}: $p_k \in L_1$ and $q_k \in L_n$ for each $k \in \N$. \\
    Assume that $p_k$ converges to $p_0$ for some $p_0 \in L_1$. 
    
    \textit{Case 2-1}: If the sequence $(d(x_n,q_k))_k$ is bounded, then we can argue as same as \textit{Case 1} since $L_n$ is proper. 
    
    \textit{Case 2-2}: If the sequence $(d(x_n,q_k))_k$ is unbounded, fix $\alpha := \sup_k d(x_n,p_k) < \infty$ and for each $k \in \mathbb{N}$ with $d(x_n,q_k) > \alpha + \delta$ we define $r_k$ to be the element in $L_n$ such that $d(x_n,r_k) = \alpha + \delta$ and $r_k \in [x_n,q_k]$ from that $L_n$ is geodesic. Since $L_n$ is proper, the sequence $(r_k) \subseteq B(x_n,\alpha+2\delta)$ converges to some $r_0 \in L_n$ passing to a subsequence. Note that $d(x_n, r_0)=\alpha+\delta$. 
    
    We claim that $|f_k(p_k)-f_k(r_k)| \geq \|f_k\|\delta$ for all such $k$. If this is the case, by letting $k \to \infty$, we obtain 
    \[
    |f(p_0)-f(r_0)| \geq \beta \delta \geq \delta.
    \]
    However, $f(r_0)=0$ since $d(x_n,r_0)=\alpha+\delta>\delta/2$, so $|f(p_0)-f(r_0)|\leq \delta/2$. This is a contradiction. 
    
To verify the claim, observe that 
\begin{align*}
|f_k(p_k) - f_k(r_k)| &\geq |f_k(p_k) - f_k(q_k) | - |f_k(q_k) - f_k(r_k)| \\
&\geq \|f_k\| d(p_k,q_k) - \|f_k\|d(q_k,r_k) \\
&\geq \|f_k\| \bigl(d(x_n,q_k) - d(p_k,x_n) - d(q_k,r_k) \bigr)  \\
& \geq \|f_k\| ( d(x_n, r_k) - \alpha ) = \|f_k\|\delta 
\end{align*}
for every $k \in \mathbb{N}$. This completes the proof.
\end{proof}

\begin{rem}
    It is natural to ask whether the conclusion of Theorem \ref{theorem:compact-proper-length} remains valid when two or more of the proper length spaces $L_1,\ldots,L_n$, satisfying $\dist(L_j,L_k)>0$ for every $j\neq k$, are allowed to be non-compact. The following example shows that this is not the case. Thus, the condition that at most one of the spaces is non-compact is essential.
    
    Let $M= \mathbb{R} \times \{0,1\} \subseteq (\mathbb{R}^2, \| \cdot\|_1)$
    and put 
    \[
    L_1 = \mathbb{R} \times \{0\} \quad \text{and} \quad L_2 = \mathbb{R}\times \{1\}.
    \]
    Then $L_1$ and $L_2$ are proper length with $\dist(L_1,L_2)=1.$ In this case, Corollary \ref{cor:disjoint-ball-p_n} shows that $\LipSE(M)$ is weakly sequentially dense in $\Lip (M)$. In fact, choose $p_n=(4n,0)$, $q_n =(4n,1)$, and $r_n = 3/2$. Then 
    \[
    d(p_n,q_n)=1 < r_n \quad \text{and} \quad \sup_n \frac{d(p_n,q_n)}{r_n} = \frac{2}{3}<1.
    \]
    It is routine to check that $(p_n,q_n)$ does not have the property (Z); hence $m_{p_n,q_n} \in \str (B_{\mathcal{F}(M)})$ for every $n \in \mathbb{N}$. 
\end{rem}

Our last result concerns the case when the set of strongly exposed points in $\F(M)$ is norm-compact. For simplicity, we denote by $\|T_f\|_{\str (B_{\F(M)})} = \sup \{ |T_{f} (\mu)|: \mu \in \str (B_{\F(M)})\}$ for $f \in \Lip (M)$.

\begin{prop}\label{prop:compact_str}
        Let $M$ be a complete metric space. If $\text{str-exp}(B_{\F(M)})$ is norm-compact, then $\|f\| = \|T_f\|_{\str (B_{\F(M)})}$ for every $f \in \overline{\LipSE (M)}^{w\text{-seq}}$.
\end{prop}

\begin{proof}
If $\str (B_{\F(M)})$ is empty, then $\LipSE (M) = \{0\}$, so there is nothing to prove. 
    If $(f_n)$ is a sequence in $\LipSE (M)$ converges weakly to some $f \in \Lip (M)$, then for any $p,q \in M$, we have 
\[
T_f (m_{p,q}) = \lim_n T_{f_n} (m_{p,q}) \leq \limsup_n \|f_n\|. 
\]
It follows that $\|f\| \leq  \limsup_n \|f_n\|$. Passing to a subsequence, we may assume that $ \lim_n \|f_n\| = \beta$ for some $\beta\geq \|f\|$ since $(f_n)$ is bounded. Moreover, weak convergence of $(f_n)$ implies pointwise convergence of $(f_n)$ on $\str (B_{\F(M)})$. By the uniform boundedness principle together with the standard compactness argument, we see that $(f_n)$ converges to $f$ uniformly on $\str (B_{\F(M)})$. Consequently, $\lim_n \|T_{f_n}-T_f\|_{\str (B_{\F(M)} )} = 0$; thus 
\[
\|f\| \leq \beta = \lim_n \|f_n\| = \lim_n \|T_{f_n}\|_{\str(B_{\F(M)}) } = \|T_f\|_{\str (B_{\F(M)})} \leq \|f\|.
\]
This shows that $\|f\| =\|T_f\|_{\str (B_{\F(M)})}$ if $f$ belongs to the weak sequential closure of $\LipSE (M)$.
\end{proof}

\begin{theorem}\label{theorem:str-exp-compact}
Let $M$ be a complete metric space and suppose that $\text{str-exp}(B_{\F(M)})$ is norm-compact. If $M = L \cup N$, where $L$ is a nontrivial length space and $N$ is a metric space with $\dist (L, N)>0$, then $\LipSE(M)$ is not weakly sequentially dense in $\Lip(M)$.
\end{theorem}

\begin{proof}
Assume that $\str (B_{\F(M)} )$ is nonempty, otherwise, the result holds trivially. 
Fix any $x_0 \in L$, and let $3\delta := d(x_0,N)>0$. Define $f \in \Lip(M)$ by 
\[
f(p) := \max\{\delta-d(p,x_0),0\} \quad (p\in M).
\]
Since $L$ is length, $\|f\|=1$. 

Let $m_{u,v} \in \str (B_{\F(M)})$. We claim that $|T_f (m_{u,v})| \leq 1/2$. 
\begin{itemize}
    \itemsep0.25em 
    \item If $u,v \in N$, then $f(u)=f(v)=0$; so $T_f(m_{u,v})=0$. 
    \item Assume that either $u$ or $v$ is in $L$, say $ u \in L$. Then $v$ cannot belong to $L$ as it contradicts $m_{u,v}\in \str (B_{\mathcal{F}(M)})$. Thus, $v \in N$, so $f(v)=0$.
    \vspace{0.25em}
    \begin{itemize}
        \itemsep0.25em
        \item If $d(x_0, u) \geq \delta$, then $f(u)=0$; so $T_f (m_{u,v})=0$.
        \item If $d(x_0,u) < \delta$, then 
        \[
        |T_f(m_{u,v})| = \frac{|f(u)-f(v)|}{d(u,v)} \leq \frac{|f(u)|}{d(v,x_0)-d(x_0,u)} \leq \frac{\delta}{3\delta-\delta} \leq \frac{1}{2},
        \]
    \end{itemize}
\end{itemize}
This proves that $\|T_f\|_{\str (B_{\F(M)})} \leq 1/2$. Thus, Proposition \ref{prop:compact_str} shows that $f$ cannot be a weak sequential limit of Lipschitz functions in $\LipSE (M)$.
\end{proof}

\begin{example}
    Let $L=[0,1]$ be equipped with the usual Euclidean metric $d_L$ and 
    \[
    T=\bigcup_{n=1}^\infty [0_{\ell_1},e_n] \subset \ell_1
    \]
    endowed with the metric $d_T$ defined by
    \[
    d_T (se_n, te_m) = \begin{cases}
        |s-t|, \,\, &n=m \\ 
        s+t, \,\, &n\neq m,
    \end{cases}
    \]
    where $(e_n)$ is the canonical basis for $\ell_1$. Let $M = L \cup T$ with distance $d$ as follows:
    \[
    d(x,y)= d_L(x,0)+ 1+ d_T(0_{\ell_1},y)
    \]
    for every $x \in L$ and $y \in T$, while $d$ coincides with $d_L$ on $L$ and with $d_T$ on $T$.

    Note that $\dist(L,T)=1$, and $T$ is a complete length space. On the other hand, $T$ is not proper since $(e_n)$ is contained in $B(0_{\ell_1},1)$ and $\|e_n-e_m\| =2$ whenever $n\neq m$. Therefore, Theorem \ref{theorem:compact-proper-length} does not apply. Nevertheless, one can check that $\str (B_{\mathcal{F}(M)}) = \{\pm m_{0,0_{\ell_1}} \}$. Hence, Theorem \ref{theorem:str-exp-compact} yields that $\LipSE(M)$ is not weakly sequentially dense in $\Lip(M)$.
\end{example}

The preceding results provide several sufficient conditions and obstructions for the weak sequential denseness of $\LipSE(M)$. However, we were unable to obtain a complete characterization of when $\LipSE (M)$ is weakly sequentially dense. We therefore record the following question.

\begin{question}
Can one characterize those metric spaces $M$ for which $\LipSE(M)$ is weakly sequentially dense in $\Lip(M)$?
\end{question}

\noindent \textbf{Acknowledgement}. 

G. Choi was supported by the National Research Foundation of Korea(NRF) grant funded by the Korea government(MSIT) (RS-2026-25475373). M. Jung was supported by June E Huh Center for Mathematical Challenges (HP086601) at Korea Institute for Advanced Study and by the research fund of Hanyang University (HY-202500000003346).

\noindent \textbf{Declarations}. \\
\indent Conflict of interest: On behalf of all authors, the corresponding author states that there is no conflict of interest.

\noindent  \textbf{Data availability}. \\
\indent No data was used for the research described in the article.

\end{document}